\documentclass[a4paper,10pt]{amsart}

\usepackage{amsmath}
\usepackage{amsthm}
\usepackage{hyperref}
\usepackage{amsfonts}
\usepackage{amssymb}
\usepackage{euscript}
\usepackage[all]{xy}
\usepackage{color}

\def\CC{\mathbb{C}}

\newcommand{\PP}{\mathbb{P}}

\newcommand{\pr}{\textrm{pr}}

\newcommand{\nc}{\newcommand}

\numberwithin{equation}{section}
\newtheorem{thm}{Theorem}[section]
\newtheorem{prop}[thm]{Proposition}
\newtheorem{lem}[thm]{Lemma}

\theoremstyle{remark}
\newtheorem{rem}[thm]{Remark}
\newtheorem{definition}[thm]{Definition}

\newtheorem*{ack}{Acknowledgments}

\nc{\hk}{\hookrightarrow}
\newcommand{\bea}{\begin{equation}}
\newcommand{\ena}{\end{equation}}

\newcommand{\Fl}{\mathcal{F}l^a}

\begin{document}
\title[Degenerate flag varieties of type A and C are Schubert varieties]{Degenerate flag varieties of type A and C \\are Schubert varieties}
\author{Giovanni Cerulli Irelli, Martina Lanini}
\address{Giovanni Cerulli Irelli:
Dipartimento di Matematica. Sapienza-Universit\`a di Roma. Piazzale Aldo Moro 5, 00185, Rome (ITALY)}
\email{cerulli.math@googlemail.com}
\address{Martina Lanini:
Department Mathematik, Universit\"at Erlangen-N\"urnberg,
Cauerstra\ss e 11, 91058 Erlangen (Germany)}
\email{lanini@math.fau.de}

\begin{abstract}
We show that any type $A$ or $C$ degenerate flag variety is in fact isomorphic to a Schubert variety in an appropriate  partial flag manifold.
\end{abstract}
\maketitle
\section{Introduction and Main Result}\label{Sec:Intro}
Appeared for the first time in the 19th Century to encode  questions in enumerative geometry, flag varieties and their Schubert varieties had been intensively studied since then, constituting  an important investigation object in topology, geometry, representation theory and algebraic combinatorics.
In the years, several variations of these varieties have been considered (affine flag and Schubert varieties, Kashiwara flag varieties, matrix Schubert varieties, toric degenerations of flags, ...). Among them, we want to focus on a class introduced recently by E.~Feigin in \cite{F2010}: the degenerate flag varieties. These are flat degenerations of  (partial) flag manifolds and turned out to be very interesting from a representation theoretic and geometric point of view. For instance, they can be used to determine a $q$-character formula for characters of irreducible modules  in type A \cite{FF, FFL} and C \cite{FFLTypeC, FFLTypeCVariety}. As for the geometry, degenerate flag varieties share several properties with Schubert varieties: they are irreducible, normal locally complete intersections with terminal and rational singularities \cite{F2010, FF, FFLTypeCVariety}. In this work we show that any degenerate flag variety of type A or C not only has a lot in common with Schubert varieties, but  it is actually isomorphic to a Schubert in an appropriate partial flag variety. In short:

\begin{thm}\label{Thm:Intro}
Degenerate flag varieties of type A and C are Schubert varieties.
\end{thm}

This result is based on the realization of degenerate flag varieties in terms of linear algebra, which is due to  E.~Feigin in type A  \cite[Theorem~2.5]{Feigin_Genocchi}  and to  E.~Feigin, M.~Finkelberg and P.~Littelmann in type C \cite[Theorem~1.1]{FFLTypeCVariety}. This description does not use any further information on the geometry of such varieties, and hence the theorem provides an independent proof of  their geometric properties such as normality, irreducibility, rational singularities, cellular decomposition, which have been established in \cite{F2010}, \cite{Feigin_Genocchi}, \cite{FF} and \cite{FFLTypeCVariety} by direct analysis.

We now state the precise version of Theorem~\ref{Thm:Intro} in the case of complete flags of type A (in Section~\ref{Sec:Parabolic} we discuss the case of partial flags, while in Section~\ref{Sec:Symplectic} we discuss the symplectic case).
Let $n\geq 1$ and $B\subset SL_{2n}$ be the subgroup of upper triangular matrices. For a weight $\lambda$ of $SL_{2n}$, let $P_{\lambda}$ be its stabilizer. Let $\omega_1,\ldots, \omega_{2n}$ be the fundamental weights and let $P:=P_{\omega_1+\omega_3+\ldots +\omega_{2n-1}}$ from now on. The Weyl group of $SL_{2n}$ is $\textrm{Sym}_{2n}$ (the symmetric group on $2n$ letters) and $P$ corresponds to the subgroup $W_J$ of $\textrm{Sym}_{2n}$ generated by the traspositions  $J=\{(2i,2i+1)_{i=1, \ldots,n-1}\}$. The variety $SL_{2n}/P$ is naturally identified with the set of partial flags
$W_1\subset W_2\subset\ldots \subset W_n$ in $\mathbb{C}^{2n}$ such that $\text{dim}(W_i)=2i-1$.

The subgroup $B$ acts on  $SL_{2n}/P$ (by left multiplication) and its orbits give  the Bruhat decomposition:
\begin{equation}\label{Eq:BruhatDecIntro}
SL_{2n}/P=\coprod_{\tau\in\textrm{Sym}_{2n}^J}B\tau P/P,
\end{equation}
where $\textrm{Sym}_{2n}^J$ is the set of permutations $\tau$ in $\textrm{Sym}_{2n}$ such that  $\tau(2i)<\tau(2i+1)$, for $i=1, \ldots, n-1$. This is the  set of minimal length representatives for the cosets in $\textrm{Sym}_{2n}/W_J$.
For a permutation $\tau\in \textrm{Sym}_{2n}^J$, let $\mathcal{C}_\tau$ be the corresponding Schubert cell in $SL_{2n}/P$, that is  $B\tau P/P$, and denote by $X_\tau=\overline{B\tau P/P }$ its closure, that is the associated Schubert variety.   Then each Schubert cell $\mathcal{C}_\tau$ has exactly one point which is fixed by the action of  the subgroup of diagonal matrices $T\subseteq B$, namely
$$\langle e_{\tau(1)}\rangle<\langle e_{\tau(1)}, e_{\tau(2)},e_{\tau(3)}\rangle<\ldots<\langle e_{\tau(1)}, e_{\tau(2)}, e_{\tau(3)}, \ldots, e_{\tau(2n-1)}\rangle.$$
(For a collection of vectors $\mathbf{v}$ of a complex vector space, we always denote by $\langle \mathbf{v}\rangle$ the subspace spanned by $\mathbf{v}$.) Let $\sigma=\sigma_n\in \textrm{Sym}_{2n}$  be the permutation defined as
\begin{equation}\label{Def:Sigma}
\sigma_n(r)=\left\{
\begin{array}{lcl}
k&\text{if}& r=2k,\\
n+1+k&\text{if}& r=2k+1.
\end{array}
\right.
\end{equation}
For example, for $n=5$ the permutation $\sigma$ is given by
$$
\left(
\begin{array}{cccccccccc}
1&2&3&4&5&6&7&8&9&10\\
6&1&7&2&8&3&9&4&10&5
\end{array}
\right).
$$
Notice that $\sigma\in\textrm{Sym}_{2n}^{J}$, indeed $\sigma(2i)=i<\sigma(2i+1)=n+1+i$ for $1\leq i\leq n-1$.

Let $\mathcal{F}l_{n+1}^a$ denote the complete degenerate flag variety associated with  $SL_{n+1}$ (see Section~\ref{Sec:Proof} for a definition of such a variety). In \cite{CFR2} it is shown that $\mathcal{F}l_{n+1}^a$  is acted upon by the maximal torus $T$ of $SL_{2n}$ (this is recalled in Section~\ref{Sec:Proof}).

We are now ready to state the precise version of Theorem~\ref{Thm:Intro} in the case of complete flags (the general result for partial flags is Theorem~\ref{Thm:Partial}).

\begin{thm}\label{Thm:Main}
There exists a T-equivariant isomorphism of projective varieties
$$
\xymatrix{
\zeta:&\Fl_{n+1}\ar^(.36){\simeq}[r]& X_\sigma\subset \textrm{SL}_{2n}/P
}
$$
where $\sigma$ is the permutation given in \eqref{Def:Sigma}
 and $P=P_{\omega_1+\omega_3+\ldots+\omega_{2n-1}}$.
\end{thm}
We notice that since the isomorphism $\zeta$ is $T$-equivariant, it is possible to compute the stalks of the local $T$-equivariant intersection cohomology of $\Fl_{n+1}$ by using the parabolic analogue of  Kazhdan-Lusztig polynomials, defined by Deodhar in \cite{Deo87}. This answers a question posed in \cite{CFR2} (and it was the original motivation for this project). Another corollary of the theorem is that the median Genocchi number $h_n=\chi(\Fl_{n+1})$ (see \cite{Feigin_Genocchi}) has another interpretation: it is the number of elements $\tau\in \textrm{Sym}_{2n}^{J}$ which are smaller than $\sigma$ in the (induced) Bruhat order.

As the anonymous referee has kindly pointed out, it would be interesting to analyze  Theorem~\ref{Thm:Intro} in the light of the recent theory of Favourable Modules \cite{FFL_Favourable}. Moreover, She/He also asks what is the interplay between the Borel subgroup $B\subset \tilde{G}$ and the ``degenerate group'' $G^a$ (see \cite{Feigin_Genocchi}), where $\tilde{G}=SL_{2n}$ (resp. $\tilde{G}=Sp_{4n-2}$) and $G=SL_{n+1}$ (resp. $G=Sp_{2n}$). Indeed both $G$ and $\tilde{G}$ act on the corresponding degenerate flag varieties (see Theorem~\ref{Thm:Main} and \ref{Thm:MainC} for the action of $\tilde{G}$ and \cite{Feigin_Genocchi, FFLTypeCVariety} for the action of $G^a$). This will be the subject of a future project.

The paper is organized as follows: in Section~\ref{Sec:Proof} we prove Theorem~\ref{Thm:Main}, in Section~\ref{Sec:Parabolic} we discuss its analogue for partial degenerate flags and in Section~\ref{Sec:Symplectic} we prove the analogous results for type C. In the Appendix we prove that the desingularization of the degenerate flag varieties provided in \cite{FF} coincides with a Bott-Samelson resolution of the Schubert variety $X_{\sigma}$.

\section{Proof of Theorem~\ref{Thm:Main}}\label{Sec:Proof}
Given an integer $n\geq1$, let $\mathcal{F}l_{n+1}^a$ denote the complete degenerate flag variety associated with  $SL_{n+1}$.  In \cite[Theorem~2.5]{Feigin_Genocchi} it is proven that $\mathcal{F}l_{n+1}^a$  can be realized as follows: let $\{f_1,\ldots, f_{n+1}\}$ be an ordered basis of a complex vector space $V\simeq \mathbb{C}^{n+1}$ and let $\pr_{k}:V\rightarrow V$ be the linear projection along the line spanned by $f_k$, i.e. $\pr_k(\sum a_if_i)=\sum_{i\neq k}a_if_i$. Then there is an isomorphism
$$
\mathcal{F}l_{n+1}^a\simeq\{(V_1,\ldots, V_n)\in\prod_{i=1}^n\textrm{Gr}_i(V)|\,\pr_{i+1}(V_i)\subset V_{i+1}\,\forall i=1,\ldots n-1\}.
$$
For convenience of notation, up to an obvious change of basis of $V$, we prefer to realize $\mathcal{F}l_{n+1}^a$ as follows:
\begin{equation}\label{Def:DegFlag}
\mathcal{F}l_{n+1}^a\simeq\{(V_1,\ldots, V_n)\in\prod_{i=1}^n\textrm{Gr}_i(V)|\,\pr_{i}(V_i)\subset V_{i+1}\,\forall i=1,\ldots, n-1\}.
\end{equation}

Let $\{e_1,\ldots,e_{2n}\}$ be an ordered basis of a vector space $W\simeq\mathbb{C}^{2n}$. For any $i=1,2,\ldots, n$, we consider the coordinate subspace $U_{n+i}:=\langle e_1,\ldots, e_{n+i}\rangle\subseteq W$ and the surjection $\xymatrix{\pi_i:U_{n+i}\ar@{->>}[r]&V}$ defined on the basis vectors as
\begin{equation}\label{Def:MapsPi}
\pi_i(e_k)=
\left\{
\begin{array}{ll}
0&\textrm{if }1\leq k\leq i-1,\\
f_k&\textrm{if }i\leq k \leq n+1,\\
f_{k-n-1}&\textrm{if }n+2\leq k\leq n+i.
\end{array}
\right.
\end{equation}
and extended by linearity to $U_{n+i}$. This induces a chain of embeddings of projective varieties
$$
\xymatrix@R=3pt{
\textrm{Gr}_{i}(V)\,\ar@{^{(}->}[r]&\textrm{Gr}_{2i-1}(U_{n+i})\,\ar@{^{(}->}[r]&\textrm{Gr}_{2i-1}(W)\\
U\ar@{|->}[r]&\pi_i^{-1}(U)\ar@{|->}[r]&\pi_i^{-1}(U)
}
$$
We call $\zeta_i:\textrm{Gr}_{i}(V)\hookrightarrow \textrm{Gr}_{2i-1}(W)$ the concatenation of the above maps. We hence have a diagonal embedding
\begin{equation}\label{Eq:DefZeta}
\xymatrix@R=3pt{
\zeta:\prod_{i=1}^n\textrm{Gr}_i(V)
\ar[r]&\prod_{i=1}^n\textrm{Gr}_{2i-1}(W)\\
(V_1, V_2, \cdots, V_n)\ar@{|->}[r]&(\zeta_1(V_1), \zeta_2(V_2), \cdots, \zeta_n(V_n))
}
\end{equation}
Let us show that $\zeta$ restricts to a map  $\Fl_{n+1}\rightarrow SL_{2n}/P$. We consider a point $(V_1, \cdots, V_n)\in\Fl_{n+1}\subset\prod_{i=1}^n\textrm{Gr}_i(V)$; thus, $\pr_i(V_i)\subset V_{i+1}$ for any $i=1,\ldots, n-1$.
We notice that $\pi_{i+1}$ coincides with $\pr_i\circ\pi_i$ on $U_{n+i}\subset U_{n+i+1}$. Denoting by $W_i:=\zeta_i(V_i)$, we get
$$
W_i\subseteq \pi_{i+1}^{-1}\pi_{i+1}(W_i)=\pi_{i+1}^{-1}\pr_i\pi_i(W_i)=\pi_{i+1}^{-1}\pr_i(V_i)\subseteq
\pi_{i+1}^{-1}(V_{i+1})=W_{i+1}.
$$
Therefore $\zeta$ restricts to an embedding $\zeta:\Fl_{n+1}\hookrightarrow SL_{2n}/P$.

We now recall the action of the maximal torus $T\subset SL_{2n}$ on $\Fl_{n+1}$ defined  in \cite[Section~3.1]{CFR}. Let $T_0$ be a maximal torus of $GL_{n+1}(\mathbb{C})$. Up to a change of basis, we assume that $T_0$ acts on $V$ by rescaling the basis vectors $f_i$'s.
This induces a diagonal action of $n$ copies $ T_0^{(1)}\times\cdots\times T_0^{(n)}$ of $T_0$ on the direct sum $V^{(1)}\oplus\cdots \oplus V^{(n)}$ of $n$ copies of $V$. More precisely we endow every copy $V^{(i)}$ with a basis $\{f_1^{(i)},\cdots, f_{n+1}^{(i)}\}$ and the torus acts by rescaling the $f_k^{(i)}$'s.  We consider the linear map $\pr_i:V^{(i)}\rightarrow V^{(i+1)}$ defined on the basis vectors by sending $f_k^{(i)}$ to $f_k^{(i+1)}$ for $k\neq i$, and $f_i^{(i)}$ to zero, and extended by linearity. We define $T_1\subset \prod_{i=1}^n T_0^{(i)}$ to be the maximal subgroup such that each projection $\pr_i:V^{(i)}\rightarrow V^{(i+1)}$ is $T_1$--equivariant. It can be checked that $T_1$ has dimension $2n$ and hence $T_1$ is isomorphic to a maximal torus of $GL_{2n}(\mathbb{C})$. More explicitly, an element $\underline{\lambda}=(\lambda_1, \ldots, \lambda_{2n})\in T_1$ acts by
\begin{equation}\label{eq:TActionFl}
\underline{\lambda}\cdot f_k^{(i)}:=\left\{\begin{array}{cc}\lambda_{k}f_k^{(i)}&\textrm{if }i\leq k\leq n+1\\
\lambda_{n+1+k}f_k^{(i)}&\textrm{if }1\leq k\leq i-1
\end{array}\right.
\end{equation}
Moreover, since the action of $T_0$ on $V$ induces an action on each Grassmannian $\textrm{Gr}_i(V)$, then the  action of $ T_0^{(1)}\times\cdots\times T_0^{(n)}$ on $V^{(1)}\oplus\cdots \oplus V^{(n)}$ induces an action of $T_1$ on the product of Grassmannians $\prod_{i=1}^n\textrm{Gr}_i(V^{(i)})=\prod_{i=1}^n\textrm{Gr}_i(V)$.  Since each projection $\pr_i$ is $T_1$-equivariant, this action descends to an action on $\Fl_{n+1}$. Notice that the action of a point $\underline{\lambda}\in T$ on $\Fl_{n+1}$ coincides with the action of any of its multiples; we hence see that $T:=T_1\cap SL_{2n}$ also acts on $\Fl_{n+1}$.

We now prove that the map $\zeta:\Fl_{n+1}\hookrightarrow SL_{2n}/P$ is $T$-equivariant. The maximal torus $T$ in $SL_{2n}$ acts on $W$ (and hence on each Grassmannian $Gr_k(W)$) by rescaling the basis vectors $e_k$'s : given $\underline{\lambda}=(\lambda_1, \ldots, \lambda_{2n})\in T$
\begin{equation}\label{Eq:TActionW}
\underline{\lambda}e_k:=\lambda_k e_k.
\end{equation}
From \eqref{eq:TActionFl} and \eqref{Eq:TActionW} it follows that each map $\pi_i$ is $T$-equivariant and hence each $\zeta_i$ is $T$-equivariant and hence $\zeta$ itself is $T$-equivariant.

We now describe the image $\zeta(\Fl_{n+1})\simeq\Fl_{n+1}$. We claim that it is given by
\begin{equation}\label{Def:Yn}
Y_n:=\left\{W_1\subset W_2\subset\ldots\subset W_n\Big |
\begin{array}{l}
\bullet \,\text{dim}W_i=2i-1\\
\bullet \, \langle e_1, e_2, \ldots, e_{i-1} \rangle\subset W_i\\
\bullet \, W_i\subset \langle e_1, \ldots, e_{n+i} \rangle
\end{array}
\right\}\subset SL_{2n}/P.
\end{equation}
Indeed, $\zeta(\Fl_{n+1})$ is clearly contained in $Y_n$; viceversa, given a flag $W_\bullet:=(W_1\subset W_2\subset\ldots\subset W_n)$ in $Y_n$, then by definition $\textrm{ker }\pi_i\subset W_i\subset U_{n+i}$ and hence $W_i=\pi_i^{-1}(\pi_i(W_i))=\zeta_i(\pi_i(W_i))$. It follows that $W_\bullet=\zeta((\pi_1(W_1), \ldots, \pi_n(W_n)))\in\textrm{Im }\zeta$. It remains to show that $(\pi_1(W_1), \ldots, \pi_n(W_n))\in \Fl_{n+1}$. This is immediately verified as follows: $\pr_i(\pi_i(W_i))=\pi_{i+1}(W_{i})\subseteq \pi_{i+1}(W_{i+1})$,
for any $i=1,\ldots, n-1$.

In order to show that $Y_n\cong X_{\sigma}$, we observe that for any $i=1, \ldots, n$ we have
$$
\#\{l\leq 2i-1 \mid\sigma(l)\leq k \}=\left\{
\begin{array}{cll}
k & \textrm{ if }&1\leq k\leq i-1,\\
i-1 & \textrm{ if }&i-1\leq k\leq n,\\
i-1+k-n & \textrm{ if }&n+1\leq k\leq n+i,\\
2i-1 & \textrm{ if }&n+i\leq k\leq 2n.
\end{array}
\right.
$$
It follows that for a partial flag $W_\bullet\in SL_{2n}/P$, condition
$
\langle e_1,e_2\ldots, e_{i-1}\rangle\subseteq W_i\subseteq \langle e_1,e_2\ldots, e_{n+i}\rangle $
is equivalent to
\begin{equation}\label{Eq:SchubCond}
 \textrm{dim}(W_i\cap \langle e_1, e_2, \ldots, e_{k} \rangle)\geq \#\{ l\leq 2i-1 \mid \sigma(l)\leq k\}
\end{equation}
for any $i=1, \ldots, n$ and $k=1, \ldots, 2n$. By \cite[Corollary of the proof of Proposition 7, \S 10.5]{Fult}, $X_{\sigma}$ is precisely the locus of partial flags in $SL_{2n}/P$ satisfying \eqref{Eq:SchubCond}. This concludes the proof of Theorem~\ref{Thm:Main}.

\begin{rem}
Theorem~\ref{Thm:Main} and its proof have a nice and clean interpretation in terms of quivers, in the spirit of \cite{CFR}, \cite{CFR2} and \cite{CFR3}.
\end{rem}

\section{Parabolic case}\label{Sec:Parabolic}

In this section we discuss the parabolic analogue of Theorem~\ref{Thm:Main}. Recall the vector space $V\simeq\mathbb{C}^{n+1}$ with basis $\{f_1,\ldots, f_{n+1}\}$ and let $\mathbf{d}=(d_i)_{i=1}^s$ be a collection of positive integers $1\leq d_1<d_2<\ldots <d_s\leq n$. For any pair of indices $1\leq i< j\leq n$ we consider the linear map $\pr_{i,j}:V\rightarrow V$ defined by $\pr_{i,j}=\pr_{j-1}\circ\ldots\circ \pr_{i+1}\circ \pr_i$ where $\pr_i$ is the projection along $f_i$ as before. Then, following \cite[Theorem~2.5]{Feigin_Genocchi}, the partial degenerate flag variety $\mathcal{F}l^a_\mathbf{d}$ is given by
$$
\mathcal{F}l^a_\mathbf{d}\simeq\{(V_1,\cdots, V_s)\in\prod_{l=1}^s\textrm{Gr}_{d_l}(V)|\, \\pr_{d_l,d_{l+1}}(V_l)\subset V_{l+1}\}.
$$
The maximal torus $T\subset SL_{2n}$ acts  on $\mathcal{F}l^a_\mathbf{d}$, in a similar way as for complete flags (see \cite{CFR}). Let $\lambda:=\omega_{2d_1-1}+\omega_{2d_2-1}+\ldots+\omega_{2d_s-1}$ and let $P=P_\lambda$ be the corresponding parabolic subgroup in $SL_{2n}$. The variety $SL_{2n}/P$ is naturally identified with the variety of partial flags $W_1\subset\cdots \subset W_{s}\subset W$ such that $\textrm{dim }W_i=2d_i-1$ ($i=1,2,\ldots, s$).
We introduce the sets $K:=\{1,2,\ldots, 2n\}\setminus\{2d_i-1|\,i=1, 2, \ldots, s\}$, $J:=\{(k,k+1)|\,k\in K\}$, and the subgroup $W_J$ of $\textrm{Sym}_{2n}$ generated by $J$.  We have the Bruhat decomposition
$$
SL_{2n}/P\simeq \coprod_{\tau}B\tau P/P
$$
where this time $\tau$ runs over the set of minimal length representatives for the cosets in  $\textrm{Sym}_{2n}/ W_J$. This set corresponds to the permutations $\tau\in \textrm{Sym}_{2n}$ such that $\tau(2d_i)<\tau(2d_i+1)<\cdots<\tau(2d_{i+1}-1)$. We denote by $X_\tau=\overline{B\tau P/P }$ the corresponding Schubert variety.  Let $\sigma_\mathbf{d}$ be the minimal length representative of the coset $\sigma_n W_J\in \textrm{Sym}_{2n}/ W_J$ ($\sigma_n$ is defined in \eqref{Def:Sigma}); explicitly, $\sigma_\mathbf{d}$ is given by
\begin{equation}\label{Def:SigmaD}
\sigma_\mathbf{d}(k)=\left\{
\begin{array}{cl}
k-d_i&\textrm{ if } k\in\{2d_i,\ldots, d_i+d_{i+1}-1\},\\
n+1+k-d_{i+1}&\textrm{ if  }k\in\{d_i+d_{i+1},\ldots, 2d_{i+1}-1\},
\end{array}
\right.
\end{equation}
with the conventions $d_0:=0$ and $d_{s+1}:=n+1$. For example, for $n=8$ and $\mathbf{d}=(2,5,7)$, we have
$$
\sigma_\mathbf{d}=
\left(
\begin{array}{cccccccccccccccc}
1&2&3&4&5&6&7&8&9&10&11&12&13&14&15&16\\
1&9&10&2&3&4&11&12&13&5&6&14&15&7&8&16
\end{array}
\right)
$$
Notice that for $\mathbf{d}=(1,2,\ldots, n)$,  the permutations $\sigma_\mathbf{d}$ and $\sigma_n$ \eqref{Def:Sigma} coincide.
\begin{thm}\label{Thm:Partial}
There exists a T-equivariant isomorphism
$$
\xymatrix{
\zeta: \mathcal{F}l^a_\mathbf{d}\ar^(.35)\simeq[r]&X_{\sigma_\mathbf{d}}\subset SL_{2n}/P_\lambda.
}
$$
\end{thm}
\begin{proof}
Recall the vector space $W\simeq\mathbb{C}^{2n}$ with basis $\{e_1, \ldots, e_{2n}\}$ and the surjections $\xymatrix@1{\pi_i:U_{n+i}\ar@{->>}[r]&V}$ defined in \eqref{Def:MapsPi} for $i=1,2,\ldots, n$.  The map $\zeta$ is defined by sending $(V_1,\cdots, V_s)\in\mathcal{F}l^a_\mathbf{d}$ to the tuple $(W_1,\cdots, W_s)\in SL_{2n}/P_\lambda$ given by
$W_i:=\pi_{d_i}^{-1}(V_i)$. It can be checked in the same way as in Section~\ref{Sec:Proof}, that the image of $\zeta$ consists of partial flags $W_1\subset W_2\subset\ldots\subset W_s$ such that $\text{dim}W_i=2d_i-1$ and $\langle e_1, e_2, \ldots, e_{d_i-1} \rangle\subseteq W_i \subseteq \langle e_1, \ldots, e_{n+d_i} \rangle$. The proof now finishes as for Theorem~\ref{Thm:Main}.
\end{proof}

\section{Symplectic case}\label{Sec:Symplectic}
In this section we state and prove the analogue of Theorem~\ref{Thm:Main} in the case of the symplectic group. In order to fix notation, we start with a brief overview about symplectic flag varieties (see e.g. \cite[Chapter 6]{LR}).
We consider a positive integer $n\geq1$  and a complex vector space $W\simeq \mathbb{C}^{2n}$ of dimension $2n$ with ordered basis $\{e_1, e_2, \ldots, e_{2n}\}$. We fix the bilinear form $b_W[\cdot, \cdot]$ on $W$ given by the following $2n\times 2n$ matrix
\begin{equation}\label{Eq:DefBilFormE}
E=\left(
\begin{array}{cc}
0&J\\
-J&0\\
\end{array}
\right)\end{equation}
where $J$ is  $n\times n$ anti-diagonal matrix with entries $(1,1, \ldots, 1)$, as usual. In particular the form is non--degenerate and  skew-symmetric. Moreover $e_k^\ast=e_{2n+1-k}$, for $k=1,\cdots, 2n$. The group $\textrm{Sp}_{2n}$ consists of those matrices $A$ in $SL_{2n}$ which leave invariant the given form, i.e. $b_W[Av,Aw]=b_W[v,w]$ for every $v,w\in W$. More explicitly, we consider the involution $\iota:\textrm{SL}_{2n}\rightarrow \textrm{SL}_{2n}$ which sends a matrix $A$ to the matrix $E({}^t A)^{-1} E^{-1}$; then the group $\textrm{Sp}_{2n}$ consists of $\iota$--invariant matrices.  The advantage of choosing the form as above is that the intersection $B\cap \textrm{Sp}_{2n}=B^\iota\subset \textrm{SL}_{2n}$ consisting of $\iota$--fixed upper triangular matrices, is indeed a Borel subgroup of $\textrm{Sp}_{2n}$ whose maximal torus is precisely the subgroup $T^\iota=T\cap\textrm{Sp}_{2n}$ of $\iota$--invariant diagonal matrices.

The parabolic subgroup $P=P_{\omega_1+\cdots+\omega_{2n-1}}$ of $SL_{2n}$ considered in Section~\ref{Sec:Intro} is mapped into itself by $\iota$ and the group of fixed points $Q:=P^\iota=P\cap \textrm{Sp}_{2n}$ is a parabolic subgroup of $Sp_{2n}$.
The projective variety $\textrm{Sp}_{2n}/Q$ can be described as follows: for a subspace $U\in\textrm{Gr}_k(W)$ we denote by $U^\perp\in \textrm{Gr}_{2n-k}(W)$ the orthogonal space of $U$ in $W$. The map
\begin{equation}\label{Def:iotaK}
\iota_k: \textrm{Gr}_k(W)\rightarrow \textrm{Gr}_{2n-k}(W):\; U\mapsto U^\perp
\end{equation}
is an isomorphism of projective varieties.  The variety $\textrm{SL}_{2n}/P$ sits inside the product $\prod_{i=1}^{n}\textrm{Gr}_{2i-1}(W)$ and we consider the involution (still denoted by $\iota$)
\begin{equation}\label{Def:iota}
\iota:=\prod_{i=1}^n\iota_{2i-1}: \prod_{i=1}^{n}\textrm{Gr}_{2i-1}(W)\rightarrow\prod_{i=1}^{n}\textrm{Gr}_{2i-1}(W)
\end{equation}
The involution $\iota$ restricts to an involution on $\textrm{SL}_{2n}/P$ and the variety $\textrm{Sp}_{2n}/Q=(\textrm{SL}_{2n}/P)^\iota$ consists of $\iota$--invariant flags.

Moreover, the involution $\iota$ (on $SL_{2n}$) induces an involution on the symmetric group $\textrm{Sym}_{2n}$ as follows: it sends $\tau\mapsto \iota(\tau)$, where $\iota(\tau)(r):=2n+1-\tau(2n+1-r)$, for $r=1, \ldots, 2n$. The Weyl group of $Sp_{2n}$ coincides with the subgroup $\textrm{Sym}_{2n}^{\iota}$ of $\iota$--fixed elements.  The element $\sigma_n\in \textrm{Sym}_{2n}$ defined in \eqref{Def:Sigma} is easily seen to be fixed by $\iota$ and it hence belongs to the Weyl group of $\textrm{Sp}_{2n}$.
The left action of $B^{\iota}$ on $Sp_{2n}/Q$ induces the Bruhat decomposition:
$$Sp_{2n}/Q=\coprod_{\tau\in(\textrm{Sym}_{2n}^J)^{\iota}}B^{\iota}\tau Q/Q.$$
Each Schubert cell $B^{\iota} \tau Q/Q$ coincides with the set of $\iota$-fixed points $\mathcal{C}_{\tau}^{\iota}$ of the  Schubert cell $\mathcal{C}_{\tau}$ of $SL_{2n}$ and the same holds for  each Schubert variety, $Z_{\tau}=\overline{B^{\iota}\tau\ Q/Q}=X_{\tau}^{\iota}$ (cf. \cite[Proposition 6.1.1.2]{LR}).

We now state the analogue of Theorem~\ref{Thm:Main} in type C. We denote by $\textrm{Sp}\mathcal{F}l_{2m}^a$ the complete degenerate flag variety associated with $\textrm{Sp}_{2m}$ (see below for a definition).

\begin{thm}\label{Thm:MainC}
There exists a $T^{ \iota}$-equivariant isomorphism of projective varieties
\begin{equation}\label{Eq:MapC}
\xymatrix{
\zeta:&\textrm{Sp}\mathcal{F}l_{2m}^a\ar^(.40){\simeq}[r]& X_{\sigma_n}^\iota\subset \textrm{Sp}_{2n}/Q
}
\end{equation}
where $n:=2m-1$, $\sigma_n$ is the permutation given in \eqref{Def:Sigma} and $Q=P^\iota$ as above.
\end{thm}

In Section~\ref{Thm:MainC} we prove Theorem~\ref{Thm:MainC} and in Section~\ref{Sec:ParabolicTypeC} we state and prove its parabolic analogue.

\subsection{Proof of Theorem~\ref{Thm:MainC}}\label{Sec:TypeCProof}
Fix an integer $m\geq1$, a complex vector space $V$ of dimension $2m$ with basis $\{f_1,\cdots, f_{2m}\}$ and a non-degenerate skew-symmetric bilinear form $b_V[\cdot, \cdot]$ on $V$ such that \begin{equation}\label{Eq:DefBilFormV}
f_{k}^\ast= \left\{\begin{array}{rcl}f_{2m-1-k} &\textrm{if} & 1\leq k\leq 2m-2,\\ f_{2m}&\textrm{if} & k=2m-1,\end{array}\right.
\end{equation}
so that $V=\langle f_1,\ldots, f_{m-1}, f_{m-1}^\ast,\ldots, f_1, f_{2m-1}, f_{2m-1}^\ast\rangle$. We define $n:=2m-1$, so that $V$ has dimension $n+1$ as in the previous sections.  The degenerate flag variety $\mathcal{F}l^a_{n+1}$ sits inside the product of Grassmannians $\prod_{i=1}^{n}\textrm{Gr}_i(V)$.  It can be checked that the map $\iota=\prod_{i} \iota_i:\prod_{i=1}^{n}\textrm{Gr}_i(V)\rightarrow \prod_{i=1}^{n}\textrm{Gr}_i(V)$ (where $\iota_i$ is defined in \eqref{Def:iotaK}) restricts to a map from  $\mathcal{F}l^a_{n+1}$ to itself, and the fixed points form the degenerate symplectic flag variety associated with $\textrm{Sp}_{2m}$ \cite[Proposition~4.7]{FFLTypeCVariety}, i.e.
\begin{equation}
\textrm{Sp}\mathcal{F}l_{2m}^a=(\mathcal{F}l_{n+1}^a)^\iota.
\end{equation}
Thus Theorem~\ref{Thm:MainC} will follow once we show that the diagram
\begin{equation}\label{Eq:CommDiagMain}
\xymatrix{
\mathcal{F}l_{n+1}^a\ar^\iota[r]\ar_\zeta[d]&\mathcal{F}l_{n+1}^a\ar^\zeta[d]\\
X_{\sigma_n}\ar^\iota[r]&X_{\sigma_n}
}
\end{equation}
commutes, where the vertical arrows denote the T-equivariant  isomorphism provided by Theorem~\ref{Thm:Main} and the horizontal arrow in the bottom is induced by the involution \eqref{Def:iota}. In  Section~\ref{Sec:Proof} we proved that the isomorphism $\zeta$ is the restriction of the map $\zeta:~\prod_{i=1}^{n}\textrm{Gr}_i(V)\rightarrow\prod_{i=1}^{n}\textrm{Gr}_{2i-1}(W)$ given in \eqref{Eq:DefZeta}. In order to prove \eqref{Eq:CommDiagMain}, it is  enough to show that the following diagram
\begin{equation}\label{Eq:DiagC}
\xymatrix{
\prod_{i=1}^{n}\textrm{Gr}_i(V)\ar^\iota[r]\ar_\zeta[d]& \prod_{i=1}^{n}\textrm{Gr}_i(V)\ar^\zeta[d]\\
\prod_{i=1}^{n}\textrm{Gr}_{2i-1}(W)\ar^\iota[r]&\prod_{i=1}^{n}\textrm{Gr}_{2i-1}(W)
}
\end{equation}
commutes. We therefore need to check that for every point $(V_i)_{i=1}^n\in\prod_{i=1}^{n}\textrm{Gr}_i(V)$ and for every $i=0,\ldots, m-1$, we have
\begin{equation}\label{Eq:Perp}
\zeta_{m-i}(V_{m-i})^\perp=\zeta_{m+i}(V_{m-i}^\perp).
\end{equation}
Recall that for every $i=1,\ldots, n$, $\zeta_i(V_i):=\pi_i^{-1}(V_i)$, where $\pi_i:U_{n+i}\rightarrow V$ is the map given in \eqref{Def:MapsPi} and $U_{n+i}$ is the coordinate subspace of W generated by $e_1,e_2,\ldots, e_{n+i}$.
We prove the following (stronger) statement: for every $i=0,\ldots, m-1$, $v\in U_{n+m-i}$ and $w\in U_{n+m+i}$ we have
\begin{equation}\label{Eq:MetricPreserving}
b_V[\pi_{m-i}(v), \pi_{m+i}(w)]=b_W[v,w].
\end{equation}
It is  easy to verify that \eqref{Eq:MetricPreserving} implies \eqref{Eq:Perp}: Indeed  $\textrm{dim }\zeta_{m-i}(V_{m-i})^\perp=2m+2i-1=\textrm{dim }\zeta_{m+i}(V_{m-i}^\perp)$ and \eqref{Eq:MetricPreserving} implies at once that $\zeta_{m+i}(V_{m-i}^\perp)\subseteq\zeta_{m-i}(V_{m-i})^\perp$. We will prove \eqref{Eq:MetricPreserving} by induction on $i\geq0$. For $i=0$ we need to show that $\pi_m:U_{n+m}\rightarrow V$ is a map of symplectic spaces, i.e. for every $v,w\in U_{n+m}$ we have $b_V[\pi_m(v),\pi_m(w)]=b_W[ v,w]$. This follows easily from the definitions: Indeed, for a given $k=1,\ldots, n$, the coordinate vector subspace $U_{n+k}$ of $W$ is given by $U_{n+k}=\langle e_1,\ldots, e_n, e_n^\ast, \ldots, e_{n-k+1}^\ast\rangle$. In particular, $U_{n+m}$ is generated by $e_1,\ldots, e_m,\ldots, e_n, e_n^\ast, \ldots, e_m^\ast$ and  $\pi_m$ is defined on the symplectic basis as follows
$$
\begin{array}{cc}
\pi_m(e_k)=\left\{
\begin{array}{cc}
0&\textrm{ if }1\leq k\leq m-1,\\
f_{n-k}^\ast&\textrm{ if } m\leq k\leq n-1,\\
f_{n}&\textrm{ if } k=n,
\end{array}
\right.\!\!\!\!\!
&
\pi_m(e_k^\ast)=\left\{
\begin{array}{cc}
f_{n-k}&\textrm{ if } m\leq k\leq n-1,\\
f_{n}^\ast&\textrm{ if } k=n.
\end{array}
\right.
\end{array}
$$

We hence assume that \eqref{Eq:MetricPreserving} is true for $i\geq0$ and we prove it for $i+1$.  In view of \eqref{Eq:DefBilFormV},  the map $\pr_{m-1+k}:V\rightarrow V$ ($1\leq k\leq m-1$) is the projection along the line spanned by the basis vector $f_{m-k}^\ast$ and we denote $\pr_{(m-k)^\ast}:=\pr_{m-1+k}$.  We notice that the adjoint map $\pr_i^\ast$ of $\pr_i:V\rightarrow V$ is $\pr_{i^\ast}$, i.e.
\begin{equation}\label{Eq:PrAdjoint}
b_V[ \pr_i(v), v']=b_V[ v,\pr_{i^\ast}(v')]
\end{equation}
for every $v,v'\in V$.
We have  already observed that $\pi_{i+1}:U_{n+i+1}\rightarrow V$ restricted to $U_{n+i}\subset U_{n+i+1}$ coincides with $\textrm{pr}_i\circ\pi_i$ and, using the notation just introduced, this means that  the following diagram
\begin{equation}\label{Eq:CommDiag}
\xymatrix@C=10pt{
V\ar^{\pr_1}[r]&V\ar^{\pr_2}[r]&\cdots\ar[r]&V\ar^{\pr_{m-1}}[r]&V\ar^{\pr_{m-1}^\ast}[r]&V\ar[r]&\ldots\ar^{\pr_2^\ast}[r]&V\ar^{\pr_1^\ast}[r]&V\\
U_{n+1}\ar^{\pi_1}[u]\ar[r]&U_{n+2}\ar^{\pi_2}[u]\ar[r]&\cdots\ar[r]&U_{n+m-1}\ar^{\pi_{m-1}}[u]\ar[r]&U_{n+m}\ar^{\pi_m}[u]\ar[r]&U_{n+m+1}\ar_{\pi_{m+1}}[u]\ar[r]&\ldots\ar[r]&U_{2n-1}\ar_{\pi_{n-1}}[u]\ar[r]&U_{2n}\ar_{\pi_n}[u]
}
\end{equation}
commutes (the chain of horizontal arrows in the bottom row is given by the canonical embeddings $U_{n+i}\hookrightarrow U_{n+i+1}$).

We can now prove \eqref{Eq:MetricPreserving}. We write a non-zero element $w\in U_{n+m+(i+1)}$ as $w=\mu e_{n-m-i}^\ast+w'$ for some $w'\in U_{n+m+i}$ and some $\mu\in\CC$;  given $v\in U_{n+m-(i+1)}$ we need to compute $b_V[ \pi_{m-(i+1)}(v), \pi_{m+(i+1)}(w) ]$. Let us first deal with the case when $w'=0$, i.e. $w=\mu e_{n-m-i}^\ast$: we have
\begin{eqnarray*}
b_V[ \pi_{m-(i+1)}(v), \pi_{m+(i+1)}(w) ]&=&\mu\, b_V[ \pi_{m-(i+1)}(v), \pi_{m+(i+1)}(e_{n-m-i}^\ast) ]\\
&=&\mu\, b_V[ \pi_{m-(i+1)}(v), f_{m+i} ]\\
&=&\mu\, b_V[ \pi_{m-(i+1)}(v), f_{m-1-i}^\ast ].
\end{eqnarray*}
By writing $v=\sum_k c_k e_k$ in the symplectic basis $\{e_k\}$, since $\pi_{m-i-1}(e_{n-m-i})=f_{n-m-i}=f_{m-1-i}$, we get
\begin{equation}\label{Eq:W'=0}
b_V[ \pi_{m-(i+1)}(v), \pi_{m+(i+1)}(w) ]= \mu c_{n-m-i}=b_W[ v, \mu e_{n-m-i}^\ast]=b_W[ v,w].
\end{equation}
We now consider the case when $w'\neq0$. In view of \eqref{Eq:PrAdjoint}, \eqref{Eq:CommDiag}, \eqref{Eq:W'=0} and the induction hypothesis we get:
\begin{eqnarray*}
b_V[ \pi_{m-(i+1)}(v), \pi_{m+(i+1)}(w) ]&&\\=b_W[ v, \mu e_{n-m-i}^\ast ]+b_V[ \pi_{m-(i+1)}(v), \pi_{m+(i+1)}(w') ]&&\\
=b_W[ v, \mu e_{n-m-i}^\ast ]+b_V[ \pi_{m-(i+1)}(v), \pr_{m-i-1}^\ast\circ \pi_{m+i}(w') ]&&\\
=b_W[ v, \mu e_{n-m-i}^\ast ]+b_W[ \pr_{m-i-1}\circ\pi_{m-i-1}(v), \pi_{m+i}(w') ]&&\\
=b_W[ v, \mu e_{n-m-i}^\ast ]+b_V[ \pi_{m-i}(v), \pi_{m+i}(w') ]&&\\
=b_W[ v, \mu e_{n-m-i}^\ast ]+b_W[ v, w' ]&&\\
=b_W[ v, w ]&&
\end{eqnarray*}
as desired.

\subsection{Parabolic case}\label{Sec:ParabolicTypeC}
We conclude by discussing the parabolic version of Theorem  \ref{Thm:MainC}, which is the type~C analogue  of Theorem~\ref{Thm:Partial}. Let $m\geq 1$ be a positive integer as in Section~\ref{Sec:TypeCProof},
and let $\mathbf{d}=(d_i)_{i=1}^s$ be a collection of positive integers $1\leq d_1<d_2<\ldots<d_s\leq 2m$  preserved by the map $d_i\mapsto 2m-d_i$. The involution $\iota=\prod\iota_i:\prod_{i=1}^{s}\textrm{Gr}_{d_i}(V)\rightarrow \prod_{i=1}^{s}\textrm{Gr}_{d_i}(V)$ is hence well-defined and restricts to a map from  $\mathcal{F}l^a_\mathbf{d}$ to itself.  The fixed points form the partial degenerate symplectic flag variety $\textrm{Sp}\mathcal{F}_\mathbf{d}^a$ \cite[Proposition~4.9]{FFLTypeCVariety}, i.e. $\textrm{Sp}\mathcal{F}_\mathbf{d}^a=(\mathcal{F}l^a_\mathbf{d})^\iota$.

Let $\lambda$ and $P_\lambda$ as in Section~\ref{Sec:Parabolic}, so that $X_{\sigma_\mathbf{d}}\subset SL_{2m}/P_{\lambda}$. Let $Q:=P_\lambda^\iota$ be the parabolic subgroup of $\textrm{Sp}_{2m}$. The projective variety $\textrm{Sp}_{2m}/Q$ coincides with the $\iota$-fixed points of $\textrm{SL}_{2m}/P_\lambda$, i.e $\textrm{Sp}_{2m}/Q=(\textrm{SL}_{2m}/P_\lambda)^\iota$. Moreover, since the permutation $\sigma_\mathbf{d}$ is fixed by $\iota$, the corresponding Schubert variety in $\textrm{Sp}_{2m}/Q$ is the variety of $\iota$-fixed points $X_{\sigma_\mathbf{d}}^\iota$ of $X_{\sigma_\mathbf{d}}$.
From the commutativity of Diagram~\eqref{Eq:DiagC}, together with Theorem \ref{Thm:Partial}, we obtain the following result.

\begin{thm}\label{Thm:PartialC}
There exists a $T^{ \iota}$-equivariant isomorphism of projective varieties
$$
\xymatrix{
\zeta:&\textrm{Sp}\mathcal{F}_\mathbf{d}^a\ar^(.35){\simeq}[r]& X_{\sigma_\mathbf{d}}^\iota\subset \textrm{Sp}_{2n}/Q
}
$$
where $\sigma_{\textbf{d}}$ is the permutation given in \eqref{Def:SigmaD}.
\end{thm}

\section*{APPENDIX: Desingularizations}

In \cite{FF} a resolution of the singularities of a type A degenerate flag variety is constructed. In view of Theorem~\ref{Thm:Main} it is natural to ask if such a desingularization coincides with a Bott-Samelson resolution of the corresponding Schubert variety. In this section, we show that this is indeed the case. In order to state and prove the result we need a combinatorial model that we discuss in Section~\ref{Sec:CombModel}. In Section~\ref{Sec:FF} we recall the construction of Feigin and Finkelberg. In Section~\ref{Sec:BottSam} we construct a Bott-Samelson resolution of $X_{\sigma_n}$. In Section~\ref{Sec:DesingMain} we prove that the two desingularizations coincide (see Theorem~\ref{Thm:Desing}).  In the whole section $n\geq 1$ is a fixed integer.

\subsection{The quiver $\Gamma_n$}\label{Sec:CombModel}
We denote by $\Gamma_n$ the quiver (i.e. the oriented graph) with $N:={n+1\choose 2}$ vertices  $\alpha_{i,j}$ with $1\leq i\leq j\leq n$ and an oriented edge $\alpha_{i,j}\rightarrow\alpha_{i+1,j}$ (for every $1\leq i<j\leq n$) and an oriented edge $\alpha_{i,j}\rightarrow\alpha_{i,j+1}$ (for $1\leq i\leq j<n$). The quiver $\Gamma_n$ is the famous Auslander-Reiten quiver of the equioriented type A quiver algebra (see e.g. \cite{ASS}).   For future reference we embed the quiver $\Gamma_n$ into the decorated quiver $\tilde{\Gamma}_n$ obtained from $\Gamma_n$ by adding $2n+1$ extra vertices $\alpha_{0,1}, \alpha_{0,2}, \alpha_{0,3},\cdots, \alpha_{0,n}, \alpha_{0,n+1}, \alpha_{1,n+1},\alpha_{2,n+2},\cdots, \alpha_{n,2n}$ and oriented edges $\alpha_{0,i}\rightarrow \alpha_{1, i}$ (for $1\leq i\leq n$) and $\alpha_{i,n}\rightarrow \alpha_{i,n+1}$ (for $1\leq i\leq n$). For example for $n=4$, the following is $\tilde{\Gamma}_4$ with the extra vertices highlighted:
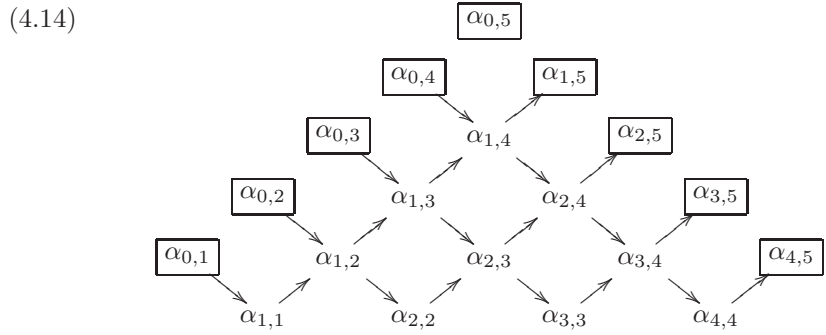
\begin{equation}\label{Fig:ARquiver}
\xymatrix@R=8pt@C=5pt{
&&&&*+[F]{\alpha_{0,5}}&&&&\\
&&&*+[F]{\alpha_{0,4}}\ar[dr]&&*+[F]{\alpha_{1,5}}&&&\\
&&*+[F]{\alpha_{0,3}}\ar[dr]&&\alpha_{1,4}\ar[dr]\ar[ur]&&*+[F]{\alpha_{2,5}}&&\\
&*+[F]{\alpha_{0,2}}\ar[dr]&&\alpha_{1,3}\ar[ur]\ar[dr]&&\alpha_{2,4}\ar[dr]\ar[ur]&&*+[F]{\alpha_{3,5}}&\\
*+[F]{\alpha_{0,1}}\ar[dr]&&\alpha_{1,2}\ar[ur]\ar[dr]&&\alpha_{2,3}\ar[ur]\ar[dr]&&\alpha_{3,4}\ar[dr]\ar[ur]&&*+[F]{\alpha_{4,5}}\\
&\alpha_{1,1}\ar[ur]&&\alpha_{2,2}\ar[ur]&&\alpha_{3,3}\ar[ur]&&\alpha_{4,4}\ar[ur]&\\
}
\end{equation}
Given an index $\ell=0,1,\cdots, 2n$, we define the \emph{$\ell$-th column of $\tilde{\Gamma}_n$} as the set of all vertices $\alpha_{i,j}$ such that $i+j-1=\ell$. For example, the vertex $\alpha_{i,i}$ lies on the column $2i-1$. Similarly, given an index $r=1, 2, \cdots, n+2$, we define the  \emph{$r$-th row of $\tilde{\Gamma}_n$} as the set of vertices $\alpha_{i,j}$ of $\tilde{\Gamma}_n$ such that $j-i+1=r$. Notice that $\Gamma_n\subset\tilde{\Gamma}_n$ has $2n-1$ columns and $n$ rows; the row at the bottom consists of the vertices
$\{\alpha_{i,i}\}_{i=1}^n$ and the top row consists of the single vertex $\alpha_{1,n}$.  Following \cite{FF} we order the vertices of $\Gamma_n$ as $\beta_1< \beta_2< \cdots <\beta_N$ so that every row is ordered from left to right and every element of a row is smaller than every element in the rows below. More precisely: we put $\beta_1:=\alpha_{1,n}$, and for $\beta_k:=\alpha_{i_k,j_k}$ we put
$$
\beta_{k+1}:=\left\{
\begin{array}{cc} \alpha_{i_k+1,j_k+1}&\textrm{ if } j_k<n,\\
\alpha_{1,n-i_k}&\textrm{ if }j_k=n.
\end{array}
\right.
$$
To illustrate: $\beta_1=\alpha_{1,n}$, $\beta_N=\alpha_{n,n}$ and $\beta_{N-k}=\alpha_{n-k,n-k}$ for $k=0,1,\cdots, n-1$. We sometimes need to extend this ordering on the vertices of $\Gamma_n$ to an ordering on the vertices of $\tilde{\Gamma}_n$. We do this in the following way: we define $\beta_0:=\alpha_{0,1}$ and
$$
\begin{array}{cc}
\beta_{-k}:=\alpha_{0,k+1},&\beta_{-n-k}:=\alpha_{k,n+1}
\end{array}
$$
for every $1\leq k\leq n$. We notice that the vertex $\beta_{-m}$ lies on the $m$--th column of $\tilde{\Gamma}_n$ (for $0\leq m\leq 2n$). We always use the following notation: the vertex $\beta_k$ is equal to $\alpha_{i_k,j_k}$ and it lies on the column $\ell_k=\ell_{\beta_k}:=i_k+j_k-1$. Moreover for a vertex $\beta_k$ of $\Gamma_n$ we denote by $\beta_k^-$ and $\beta_k^+$ the vertices of $\tilde{\Gamma}_n$ given by:
\begin{equation}\label{Eq:Betapm}
\begin{array}{cc}
\beta_k^-:=\alpha_{i_k-1,j_k}&\beta_k^+:=\alpha_{i_k,j_k+1}.
\end{array}
\end{equation}
By construction, both $\beta_k^-$ and $\beta_k^+$ are \emph{strictly smaller} than $\beta_k$ and $\ell_{\beta_k^\pm}=\ell_{\beta_k}\pm 1$.
We will also need the following notation: for a vertex $\beta_k$ of $\Gamma_n$ and a column index $\ell\in[1,2n-1]$ of $\Gamma_n$ we define the index
\begin{equation}\label{Eq:Def(B:t)}
(\beta_k: \ell)=\textrm{max}\{\beta_s\,|-2n\leq s\leq k,\, \beta_s\textrm{ lies on the }\ell\textrm{-th column}\}.
\end{equation}
The following equalities are a direct consequence of the definition:
\begin{eqnarray}
(\beta_k:\ell_k)&=&\beta_k\\
\label{Eq:(Beta:ell-1)}
(\beta_k:\ell_k\pm1)&=&\beta_k^\pm\\
\label{Eq:(Aij:j)}
(\alpha_{1,j}:\ell)&=&\beta_{-\ell}\;\;\;\forall\,0\leq\ell<j,\\\label{Eq:(BetaN:ell)}
(\beta_N:2k-1)&=&\beta_{N-(n-k)}\;\;\;\forall\,1\leq k\leq n\\\label{Eq:(betaK:t)=(betaK-1:t)}
(\beta_k:t)&=&(\beta_{k-1}:t)\;\;\;\forall t\neq\ell_k. 
\end{eqnarray}
The following technical result will be used later.
\begin{lem}\label{Lemma:(Beta:t)^+}
For $1\leq k\leq N$ and $1\leq t\leq 2n-2$,  either $(\beta_k:t+1)=(\beta_k:t)^+$ or $(\beta_k:t+1)^-=(\beta_k:t)$.
\end{lem}
\begin{proof}
Let $\beta_s:=(\beta_k:t)$ and $\beta_r:=(\beta_k:t+1)$. Then, $\beta_s$ and $\beta_r$ lie in different columns and  either $\beta_s<\beta_r$ or $\beta_r<\beta_s$. If $\beta_s<\beta_r$ then $\beta_r^-=\beta_s$ by maximality of $\beta_s$. If $\beta_r<\beta_s$ then $\beta_r=\beta_s^+$ by maximality of $\beta_r$.
\end{proof}
\subsection{The desingularization of Feigin and Finkelberg}\label{Sec:FF}

In this section we recall the construction of a desingularization $\xymatrix@1{R_n\ar@{->>}[r]&\mathcal{F}l_{n}^a}$ of the complete degenerate flag variety $\mathcal{F}l_n^a$ due to E.~Feigin and M.~Finkelberg \cite{FF}.
We fix a complex vector space $V$ of dimension $n+1$ and ordered basis $\{f_1,\cdots, f_{n+1}\}$. For every $1\leq k\leq N$ we consider the coordinate subspace $V_{\beta_k}$ of V given by
\begin{equation}\label{Def:Vij}
V_{\beta_k}:=\textrm{Im }(\pr_{j_k-1}\circ\cdots\circ\pr_{i_k})=\langle f_1,\ldots, f_{i_k-1},f_{j_k},\ldots, f_{n+1}\rangle
\end{equation}
with the convention that $f_0=0$ so that $V_{1,j}=\langle f_j,\cdots, f_{n+1}\rangle$ (as before, the parenthesis $\langle \cdots \rangle$ denotes the $\mathbf{C}$-linear span of the collection of vectors enclosed within them and $\pr_k:V\rightarrow V$ is the projection along the line generated by $f_k$). Notice that the coordinate subspace $V_{\beta_k}$ is the analogous of $W_{i_k,j_k}$ in \cite{FF} taking into account our notation \eqref{Def:DegFlag}. We decorate the collection $(V_{\beta_k})_{k=1}^N$ with the subspaces
$$
\begin{array}{ccc}
V_{\beta_0}=V_{\beta_{-k}}:=\{0\}& V_{\beta_{-n-k}}:=V
\end{array}
$$
for every $k=1,\cdots, n$.  
In \cite[Definition~2.1]{FF},  the variety $R_n$ is defined as the variety of collections $(Z_{\beta_k})_{k=1}^N$ such that
\begin{eqnarray}\label{Eq:CondRn1}
&&Z_{\beta_k}\in \textrm{Gr}_{i_k}(V_{\beta_k}),\\\label{Eq:CondRn2}
&&Z_{\beta_k^-}\subset Z_{\beta_k},\\\label{Eq:CondRn3}
&&\textrm{pr}_{j_k}(Z_{\beta_k})\subset Z_{\beta_k^+},
\end{eqnarray}
with the convention that $Z_{0,k}=\{0\}$ and $Z_{k,n+1}=V$ for $1\leq k\leq n$. The variety $R_n$ is a tower of $\mathbf{P}^1$-fibrations: for every $1\leq s\leq N$ we consider the variety $R_n(s)$ of ``truncated'' collections $(Z_{\beta_k})_{k=1}^s$ satisfying conditions \eqref{Eq:CondRn1}, \eqref{Eq:CondRn2}, \eqref{Eq:CondRn1}. Then there is a natural epimorphism
$$
p_n(s):\xymatrix{R_n(s+1)\ar@{->>}[r]&R_n(s)}:\, (Z_{\beta_i})_{1\leq i\leq s+1}\mapsto(Z_{\beta_i})_{1\leq i\leq s}.
$$
whose fiber is $\PP^1$. In particular $R_n$ is a smooth projective variety and the map
$$
p_n: R_n\rightarrow \mathcal{F}l^a_{n+1}: (Z_{\beta_k})_{k=1}^N\mapsto (Z_{i,i})_{1\leq i\leq n}
$$
is a desingularization of $\mathcal{F}l^a_{n+1}$ (\cite{FF}).

\subsection{A Bott-Samelson resolution of $X_{\sigma_n}$}\label{Sec:BottSam}

In this Section we discuss a Bott-Samelson resolution of the Schubert variety $X_{\sigma_n}$ (see \eqref{Def:Sigma}).
We choose the following reduced expression for the permutation $\sigma_n\in\textrm{Sym}_{2n}$ :
\begin{equation}\label{Eq:SigmaReduced}
\sigma_n=s_n\circ(s_{n-1}\circ s_{n+1})\circ(s_{n-2}\circ s_n\circ s_{n+2})\circ\cdots\circ(s_1\circ s_3\circ\cdots\circ s_{2n-1})
\end{equation}
($s_i$ denotes the simple transposition $(i,i+1)$). In particular we see that the length of $\sigma_n$ is $N={n+1\choose 2}$ and the simple reflections appearing in \eqref{Eq:SigmaReduced} are in bijection with the vertices of $\Gamma_n$ as follows
$$
\sigma_n=\tau_{\beta_1}\circ\tau_{\beta_2}\circ\cdots\circ\tau_{\beta_N}.
$$
To illustrate: $\tau_{\beta_1}=s_n$, $\tau_{\beta_2}=s_{n-1}$, $\tau_{\beta_N}=s_{2n-1}$. By construction, 
$\tau_{\beta_k}=s_{\ell_k}$.

We fix a complex vector space $W$ of dimension $2n$ with ordered basis $\{e_1,\cdots,e_{2n}\}$.
Given two complete flags
$$
U_\bullet=(U_1\subset U_2\subset\cdots\subset U_{2n-1}\subset W),\,W_\bullet=(W_1\subset W_2\subset\cdots\subset W_{2n-1}\subset W)
$$
in $SL_{2n}/B$ and an integer $k=1,2,\cdots, 2n-1$, the pair $(U_\bullet,W_\bullet)$ is said to be in \emph{relative position} $k$ if
$U_i=W_i$ for every $i\neq k$. Let $F_\bullet$ be the standard flag of $W$, i.e. $F_i=\langle f_1,\ldots, f_i\rangle$, for every $0\leq i\leq 2n$ ($F_0:=\{0\}$). Notice that the subspace $F_{n+i}$ was denoted by $U_{n+i}$ before and we freely use the two notations. We consider the variety
\begin{equation}\label{Def:BS}
\textrm{BS}_n:=\{(U^{\beta_k}_\bullet)_{k=0}^N\in (SL_{2n}/B)^{N+1}|\begin{array}{l}
\,U^{\beta_0}_\bullet=F_\bullet\textrm{ and }(U^{\beta_{k+1}}_\bullet,U^{\beta_{k}}_\bullet)\textrm{ is }\\\textrm{ in relative position }\ell_{k+1}
\end{array}\}.
\end{equation}
This variety is the Bott-Samelson variety associated with the reduced expression \eqref{Eq:SigmaReduced} of $\sigma_n$. 
In particular, the following result is well-known  \cite{Hansen}, \cite{Demazure}.  
\begin{prop}
The variety $\textrm{BS}_n$ is smooth. The map
\begin{equation}\label{Eq:DesingMapBS}
\xymatrix@R=5pt{
\rho_n:BS_n\ar@{->>}[r]&X_{\sigma_n}:
(U^{\beta_k}_\bullet)_{k=0}^N\ar@{|->}[r]&(U^{\beta_N}_{2k-1})_{k=1}^n
}
\end{equation}
is a desingularization.
\end{prop}

We now give a description of $\textrm{BS}_n$ which will be used later.
\begin{definition}\label{Def:Bn}
Let $B_n$ be the projective variety of collections $(U_{\beta_k})_{k=-2n}^N$ (indexed by the vertices of $\tilde{\Gamma}_n$) such that, for every $1\leq k\leq N$ and $0\leq t\leq 2n$,
\begin{eqnarray}
\label{Def:BnCond3}
&&U_{\beta_{-t}}=F_t,\\\label{Def:BnCond1}
&&U_{\beta_k}\in \textrm{Gr}_{\ell_k}(W),\\\label{Def:BnCond2}
&&U_{\beta_k}\subset U_{n+i_k},\\\label{Def:BnCond4}
&&U_{\beta_k^-}\subset U_{\beta_k}\subset U_{\beta_k^+}.
\end{eqnarray}
\end{definition}
\begin{prop}\label{Prop:BSnBn}
The map $\theta_n:\textrm{BS}_n\rightarrow B_n$ which sends a point $U_\bullet^{\beta_\bullet}=(U^{\beta_k}_\bullet)_{k=0}^N$ to the collection $\theta_n(U_\bullet^{\beta_\bullet})=(U_{\beta_k})_{k=-2n}^N$ defined by
$$
U_{\beta_k}:=\left\{
\begin{array}{cc}
U^{\beta_k}_{\ell_k}&\textrm{ if }k\geq0\\&\\
U^{\beta_0}_{k}&\textrm{ if }k<0
\end{array}
\right.
$$
is an isomorphism of projective varieties. The inverse $\psi=\theta_n^{-1}$ sends a point $(U_{\beta_k})_{k=-2n}^N$ of $B_n$ to the collection of complete flags $(U^{\beta_k}_\bullet)_{k=0}^N$ defined by
\begin{equation}\label{Eq:InverseTheta}
U^{\beta_k}_t:=U_{(\beta_k:t)}
\end{equation}
for every $1\leq t\leq 2n-1$.
\end{prop}
\begin{proof}

Let us prove that the map $\theta_n$ is well-defined. This is based on the following technical result. 
\begin{lem}\label{Lem:UBetaDIm}
Let $(U^{\beta_k}_\bullet)_{k=0}^N$ be a point of $\textrm{BS}_n$. Then,
\begin{enumerate}
\item For every $1\leq k\leq N$, and $1\leq t\leq 2n-1$ we have
\begin{equation}\label{Eq:UBeta^t=U(beta:T)}
U^{\beta_k}_t=U^{(\beta_k:t)}_t
\end{equation}
\item For $1\leq k\leq N$
\begin{equation}\label{Eq:F^jSubset UBetak}
F_{j_k-1}\subset U^{\beta_k}_{\ell_k}\subset F_{n+i_k}
\end{equation}
\end{enumerate}
\end{lem}
\begin{proof}
By definition,  $U^{\beta_k}_t=U^{\beta_{k-1}}_t$ for $t\neq\ell_k$ and, in particular, \eqref{Eq:UBeta^t=U(beta:T)} holds for $k=1$. By induction on $k\geq 1$, using \eqref{Eq:(betaK:t)=(betaK-1:t)}, we get the desired \eqref{Eq:UBeta^t=U(beta:T)} . 

If $\beta_k=\alpha_{1,j_k}$ then $\ell_k=j_k$ and $(\beta_k:j_k-1)=\beta_{-(j_k-1)}$ (see \eqref{Eq:(Aij:j)}); in view of \eqref{Eq:UBeta^t=U(beta:T)} we have $U^{\beta_k}_{\ell_k-1}=U^{\beta_0}_{\ell_k-1}=F_{\ell_k-1}$ and hence $F_{j_k-1}=U^{\beta_k}_{\ell_k-1}\subset U^{\beta_k}_{\ell_k}$.  Similarly, if $\beta_k=\alpha_{i,n}$, then  $\ell_k=n+i-1$ and $(\beta_k:\ell_k+1)=\beta_k^+=\beta_{-(n+i)}$; in view of \eqref{Eq:UBeta^t=U(beta:T)},  $U^{\beta_k}_{\ell_k+1}=U^{\beta_0}_{\ell_k+1}=F_{n+i}$ and hence $U^{\beta_k}_{\ell_k}\subset U^{\beta_k}_{\ell_k+1}=F_{n+i}$. We hence assume that both $\beta_k^-$ and $\beta_k^+$ are vertices of $\Gamma_n$. In this case we have
$$
U^{\beta_k^-}_{\ell_k-1}=U^{\beta_k}_{\ell_k-1}\subset U^{\beta_k}_{\ell_k}\subset U^{\beta_k}_{\ell_k+1}=U^{\beta_k^+}_{\ell_k+1}.
$$
By induction on the row index, (using \eqref{Eq:Betapm}) the desired \eqref{Eq:F^jSubset UBetak} follows.
\end{proof}
Given $U_\bullet^{\beta_\bullet}\in\textrm{BS}_n$, in view of Lemma~\ref{Lem:UBetaDIm} we have $U^{\beta_k}_{\ell_k}\in \textrm{Gr}_{\ell_k}(U_{n+i_k})$; moreover
$$
U^{\beta_k^-}_{\ell_{\beta_k^-}}=U^{\beta_k}_{\ell_k-1}\subset U_{\beta_k}\subset U^{\beta_k}_{\ell_k+1}=U^{\beta_k^+}_{\ell_{\beta_k^+}}
$$
and hence $\theta_n(U_\bullet^{\beta_\bullet})$ satisfies all the four conditions \eqref{Def:BnCond3}--\eqref{Def:BnCond4} and the map $\theta_n$ is well-defined and it is clearly algebraic. 

With the help of Lemma \ref{Lemma:(Beta:t)^+} and in view of \eqref{Eq:UBeta^t=U(beta:T)}, one easily checks that the map $\psi$  is well-defined. It is clearly algebraic and the composition with $\theta_n$ is the identity.
\end{proof}

\subsection{Main Result}\label{Sec:DesingMain}
We are now ready to state and prove the main result of the this appendix.

\begin{thm}\label{Thm:Desing}
There exists an isomorphism $\xymatrix@1@R=5pt{\psi_n:R_n\ar[r]&\textrm{BS}_n}$
of projective varieties such that the following diagram
\begin{equation}\label{DiagramThmDesing}
\xymatrix{
R_n\ar^\psi[r]\ar@{->>}_{p_n}[d]&\textrm{BS}_n\ar@{->>}^{\rho_n}[d]\\
\mathcal{F}l_{n+1}^a\ar^\zeta[r]&X_{\sigma_n}
}
\end{equation}
commutes (where the map $\zeta$ is the one in Theorem~\ref{Thm:Main}).
\end{thm}
\begin{proof}
Recall that for every $1\leq i\leq j\leq n$ we have a commutative diagram
$$
\xymatrix@C=50pt{
U_{n+i}\ar_{\pi_i}[d]\ar@{^{(}->}[r]&U_{n+j}\ar^{\pi_j}[d]\\
V\ar^{\pr_{j-1}\circ\cdots\circ \pr_i}[r]&V
}
$$
and it follows immediately from \eqref{Def:Vij} that
\begin{equation}\label{Eq:Vbeta}
V_{\beta_k}=\pi_{j_k}(U_{n+i_k}).
\end{equation}
Since $\textrm{dim Ker }(\pi_{j_k})=j_k-1$ and hence $\ell_k=i_k+\textrm{dim Ker }(\pi_{j_k})$, we have an embedding of projective varieties:
$$
\xymatrix@R=5pt{
\zeta_{\beta_k}:\ar@{^{(}->}[r]\textrm{Gr}_{i_k}(V_{\beta_k})&\textrm{Gr}_{\ell_k}(U_{n+i_k}),\\
Z\ar@{|->}[r]&\pi_{j_k}^{-1}(Z).
}
$$
This induces an embedding
$$
\xymatrix@R=5pt{
\zeta=\zeta_n:=\prod_{k=1}^N\zeta_{\beta_k}:\prod_{k=1}^N\textrm{Gr}_{i_k}(V_{\beta_k})\ar@{^{(}->}[r]&\prod_{k=1}^N\textrm{Gr}_{\ell_k}(U_{n+i_k})}.
$$
In view of \eqref{Eq:CondRn1},  the variety $R_n$ is contained in $\prod_{k=1}^N\textrm{Gr}_{i_k}(V_{\beta_k})$ and we claim that the image $\zeta(R_n)$ is canonically isomorphic to the variety $B_n$ (see Definition \ref{Def:Bn}).
Indeed let $U=(U_{\beta_k})_{k=1}^N=\zeta((Z_{\beta_k})_{k=1}^N)\in \zeta(R_n)$, i.e. $U_{\beta_k}:=\pi_{j_k}^{-1}(Z_{\beta_k})$. We decorate the collection $U=(U_{\beta_k})_{k=1}^N$ with the subspaces
$U_{\beta_{-t}}:=F_t$, for $0\leq t\leq 2n$.  The decorated collection $\tilde{U}=(U_{\beta_k})_{k=-2n}^N$ satisfies \eqref{Def:BnCond3}, \eqref{Def:BnCond1}, \eqref{Def:BnCond2} and \eqref{Def:BnCond4}, i.e. $\tilde{U}\in B_n$.  Indeed, \eqref{Def:BnCond3} is clearly satisfied, \eqref{Def:BnCond1} follows from \eqref{Eq:CondRn1} and \eqref{Def:BnCond2} follows from  \eqref{Eq:CondRn2}. Since $Z_{\beta_k}\subset V_{\beta_k}=\pi_{j_k}(U_{n+i_k})$, there exists $U\subset U_{n+i_k}$ such that $Z_{\beta_k}=\pi_{j_k}(U)$; in view of \eqref{Eq:CondRn3}, $\pr_{j_k}(Z_{\beta_k})\subset Z_{\beta_k^+}$ and we have
$$
\pi_{j_k+1}^{-1}(Z_{\beta_k^+})\supset \pi_{j_k+1}^{-1}\pr_{j_k}(Z_k)=\pi_{j_k+1}^{-1}\pi_{j_k+1}(U)\supseteq U=\pi_{j_k}^{-1}(Z_k)
$$
and hence  $\tilde{U}$ satisfies \eqref{Def:BnCond4}. 

On the other hand,  let $(U_{\beta_k})_{k=-2n}^N$ be a point of $B_n$. We define $Z_{\beta_k}:=\pi_{j_k}(U_{\beta_k})$ ($k=1,2,\cdots, N$) and we check that $(Z_{\beta_k})_{k=1}^N\in R_n$, i.e. it satisfies \eqref{Eq:CondRn1}, \eqref{Eq:CondRn2} and \eqref{Eq:CondRn3}. In view of \eqref{Eq:Vbeta}, $Z_{\beta_k}\in V_{\beta_k}$. Since $U_{\beta_k^-}\subset U_{\beta_k}$  we have $\textrm{Ker }\pi_{j_k}=F_{j_k-1}\subseteq U_{\beta_k}$ and hence $\textrm{dim }Z_{\beta_k}=\ell_k-(j_k-1)=i_k$. It follows that $Z_{\beta_k}\in\textrm{Gr}_{i_k}(V_{\beta_k})$ and \eqref{Eq:CondRn1} is satisfied. Since $U_{\beta_k^-}\subset U_{\beta_k}$, \eqref{Eq:CondRn2} follows. To prove \eqref{Eq:CondRn3}, we notice that
$$
\pr_{j_k}(Z_{\beta_k})=\pr_{j_k}\circ \pi_{j_k}(U_{\beta_k})=\pi_{j_{k+1}}(U_{\beta_k})\subseteq \pi_{j_{k+1}}(U_{\beta_k^+})=Z_{\beta_k^+}
$$
and the claim is proved.

In view of Proposition \ref{Prop:BSnBn}, the  map $\psi_n:=\theta_n^{-1}\circ\zeta_n:R_n\rightarrow \textrm{BS}_n$ is hence an isomorphism of projective varieties. It remains to check that the diagram \eqref{DiagramThmDesing} commutes. Using \eqref{Eq:(BetaN:ell)}, we get 
$$
\rho_n\circ\psi_n((Z_{\beta_k})_{k=1}^N)=(\pi_{j_{k}}^{-1}(Z_{\alpha_{k,k}}))_{k=1}^{n}.
$$
On the other hand, $p_n((Z_{\beta_k})_{k=1}^N)=(Z_{\alpha_{k,k}})_{k=1}^n$, and
$$
\zeta\circ p_n((Z_{\beta_k})_{k=1}^N)=(\pi_{j_k}^{-1}(Z_{\alpha_{k,k}}))_{k=1}^{n}
$$
as desired.
\end{proof}

\begin{ack} The main part of this project took place during M.L. stay at the Department of Mathematics ``Guido Castelnuovo'' of ``Sapienza-Universit\`a di Roma''. We thank that institution for the perfect working conditions. The work of M.L. was financed  by DFG SPP1388  and ``Teoria delle rappresentazioni e applicazioni, Progetto di Ateneo 2012, Sapienza Universit\`a di Roma''. The work of G.C.I. was financed by the national FIRB grant RBFR12RA9W.  We thank  Rocco Chiriv\`i, Francesco Esposito and Paolo Bravi for helpful discussions on a previous version of the present paper. We are especially grateful to Corrado De Concini, Michael Finkelberg and Evgeny Feigin for their useful comments and to Peter Littelmann and Oksana Yakimova for several discussions about possible applications of this work.
\end{ack}

\bibliographystyle{amsplain}

\end{document}